\documentclass[12pt]{amsart}
\usepackage{graphicx}

\newtheorem{thm}{Theorem}
\newtheorem*{thm*}{Theorem}
\newtheorem{cor}{Corollary}
\newtheorem{lem}{Lemma}

\newtheorem{rem}{Remark}

\newcommand{\as}{\text{ as }}
\newcommand{\ninf}{n\rightarrow\infty}

\renewcommand{\Re}{\text{Re}}
\renewcommand{\Im}{\text{Im}}
\DeclareMathOperator{\C}{\mathbb{C}}
\DeclareMathOperator{\R}{\mathbb{R}}
\renewcommand{\l}{\left}
\renewcommand{\r}{\right}

\begin{document}

\title[Weyl-Titchmarsh type formula]{Weyl-Titchmarsh type formula for discrete Schr\"odinger operator with Wigner-von Neumann potential}

\author{Jan Janas}
    \address{Instytut Matematyczny Polskiej Akademii Nauk, ul. \'sw. Tomasza 30, 31-027 Krak\'ow, Poland}
    \email{najanas@cyf-kr.edu.pl}

\author{Sergey Simonov}
    \address{Department of Mathematical Physics, Institute of Physics, St. Petersburg University,
    Ulianovskaia 1, St. Petergoff, St. Petersburg, 198904, Russia}
    \email{sergey\_simonov@mail.ru}

\date{}

\begin{abstract}
We consider discrete Schr\"odinger operator $\mathcal J$ with Wigner-von Neumann potential not belonging to $l^2$.
We find asymptotics of orthonormal polynomials associated to $\mathcal J$.
We prove the Weyl-Titchmarsh type formula, which relates the spectral density of $\mathcal J$ to a coefficient in asymptotics of orthonormal polynomials.
\end{abstract}

\maketitle

\section{Introduction}
In recent papers on Jacobi matrices
\cite{Damanik-Simon-06},\cite{Janas-Naboko-Sheronova-07},\cite{Janas-Moszynski-06},\cite{Naboko-Simonov-10},\cite{Naboko-Pchelintseva-Silva-09},\cite{Moszynski-09}
new results were found on the asymptotics of generalized eigenvectors of these
operators. For real sequences $\{a_n\}_{n=1}^{\infty}$ and $\{b_n\}_{n=1}^{\infty}$ the Jacobi operator $J=J(a_n,b_n)$ is defined in the Hilbert space $l^2$ by the formula
$$ (Ju)_n = a_{n-1}u_{n-1}+ b_nu_n + a_nu_{n+1} .$$
As it is well known the spectral analysis of Jacobi operators is strongly related to the study of the asymptotics
of generalized eigenvectors. In this work we concentrate on the discrete Schr\"odinger operator $\mathcal J = \mathcal J(1,b_n)$ with the Wigner-von Neumann potential

\begin{equation}\label{eq entries}
    b_n=\frac{c\sin(2\omega n+\delta)}{n^{\gamma}}+q_n,
\end{equation}
\begin{equation}\label{eq conditions}
    \gamma\in\l(\frac13;\frac12\r),2\omega\notin\pi\mathbb Z\text{ and }\{q_n\}_{n=1}^{\infty}\in l^1,
\end{equation}
where $c,\omega,\delta,q_n\in\R$.
$\mathcal J$ is the Jacobi operator given by
\begin{equation}\label{eq definition of the operator}
    \begin{array}{l}
    (\mathcal{J}u)_1=b_1u_1+u_2, \\
    (\mathcal{J}u)_n=u_{n-1}+b_nu_n+u_{n+1},\ n\ge2. \\
    \end{array}
\end{equation}
and has a matrix representation in the canonical basis $\{e_n\}_{n=1}^{\infty}$ of $l^2$
\begin{equation*}
    \mathcal J=
    \left(%
    \begin{array}{cccc}
    b_1 & 1 & 0 & \cdots \\
    1 & b_2& 1 & \cdots \\
    0 & 1 & b_3 & \cdots \\
    \vdots & \vdots & \vdots & \ddots \\
    \end{array}%
    \right).
\end{equation*}
Since $\mathcal J$ is a compact perturbation of the free discrete Schr\"odinger operator its essential spectrum is the interval $[-2;2]$.
Frequency $2\omega$ in the potential produces (in general) four critical (or resonance) points inside this interval: $\pm2\cos\omega,\pm2\cos2\omega$.
At these points the resonance occurs and the asymptotics of the generalized eigenvectors changes
(analogous phenomenon in the continuous case is very well studied, see \cite{Harris-Lutz-75}, \cite{Eastham-89}, Chapter 4 and \cite{Kurasov-Naboko-07}, Theorem 5)
and an eigenvalue can appear under certain additional conditions.

In the present paper we are interested in asymptotics of orthogonal polynomials $P_n(\lambda)$ associated to $\mathcal J$, which are
\begin{equation*}
    \begin{array}{l}
      P_1(\lambda):=1,P_2(\lambda):=\lambda-b_1, \\
      P_{n+1}(\lambda):=(\lambda-b_n)P_n(\lambda)-P_{n-1}(\lambda),\ n\ge2,
    \end{array}
\end{equation*}
and the relation of this asymptotics to the spectral density $\rho'(\lambda)$ of $\mathcal J$.

Our main result (see Theorem \ref{thm asymptotics of polynomials} on page \pageref{thm asymptotics of polynomials}
and Theorem \ref{thm spectral density} on page \pageref{thm spectral density} for the exact formulation)
is that there exists a function $F$ ( the Jost function ) such that for a.a. $\lambda\in(-2;2)$,
\begin{equation}\label{eq spectral density intro}
    \rho'(\lambda)=\frac{\sqrt{4-\lambda^2}}{2\pi|F(z)|^2}
\end{equation}
and
\begin{multline*}
    P_n(\lambda)=\frac{zF(z)}{1-z^2}\frac1{z^n}\exp\l(\frac{\mu_2(z)n^{1-2\gamma}}{1-2\gamma}\r)
    \\
    +\frac{z\overline{F(z)}}{z^2-1}z^n\exp\l(-\frac{\mu_2(z)n^{1-2\gamma}}{1-2\gamma}\r)
    +o(1)\as\ninf,
\end{multline*}
where
\begin{equation*}
    \lambda=z+\frac1z,\ z=\frac{\lambda-i\sqrt{4-\lambda^2}}2
\end{equation*}
and
\begin{equation}\label{eq mu-2}
    \mu_2(z):=\frac{c^2z^2(1+z^2)e^{2i\omega}}{4(1-z^2)(z^2-e^{2i\omega})(1-z^2e^{2i\omega})}.
\end{equation}
This result is an analog of the classical Weyl-Titchmarsh (or Kodaira) formula for the differential Schr\"odinger operator on the half-line with summable potential
(see \cite{Titchmarsh-46-1}, Chapter 5 and \cite{Kodaira-49}).

We consider $\gamma\in\l(\frac13;\frac12\r)$ for the following reasons. If $\gamma>\frac12$, then $\{b_n\}_{n=1}^{\infty}\in l^2$ and the sum
\begin{equation*}
    \sum_{n=1}^{\infty}b_n
\end{equation*}
converges. This situation was studied using a completely different
method by Damanik and Simon in \cite{Damanik-Simon-06} for a class of more general Jacobi matrices, with arbitrary sequence $\{b_n\}_{n=1}^{\infty}\in l^2$ for which the series $\sum_{n=1}^{\infty}b_n $ is convergent.
In this case the asymptotics of polynomials $P_n(\lambda)$ is simpler (of Szeg\"o type):
\begin{equation}\label{eq asymptotics of P Damanik-Simon case}
    P_n(\lambda)=\frac{zF(z)}{1-z^2}\cdot\frac1{z^n}+\frac{z\overline{F(z)}}{z^2-1}\cdot z^n+o(1)\as\ninf
\end{equation}
, and this was proven in \cite{Damanik-Simon-06} not for a.a. $\lambda\in(-2;2)$, but in the sense of the convergence in $L_2((-2;2),\rho'(x)dx)$ (Theorem 8.1).
If one proves that \eqref{eq asymptotics of P Damanik-Simon case} holds for a.a. $\lambda\in(-2;2)$ (which for the potential of the form \eqref{eq entries}
is much simpler than the analysis that we develop in the present paper), then \eqref{eq spectral density intro} can be deduced from results of Damanik and Simon
(Theorem 5.6 and Theorem 8.1) in a non-trivial way (this was pointed out by Dr. Roman Romanov in private communication).
On the other hand, the type of asymptotics for $\gamma\in(\frac12;1]$ is the same as in the simple case $\gamma>1$.
For $\gamma=\frac12$ the type of asymptotics changes. This happens also for every $\gamma=\frac1l$, $l\in\mathbb N\backslash\{1\}$.
We are forced to consider different (depending on $l$) number of terms in the asymptotic expansion that we use, see \eqref{eq T-1}.
The greater $l$ is, the more terms are significant. The method that we use works for every $\gamma>0$, but we restrict ourselves to the case
$\gamma\in(\frac13;\frac12)$ to show how it works in the general case. However, in the final section we state the corresponding result for $\gamma\in(\frac12;1]$
(Theorem \ref{thm result in simpler case}) and indicate how to simplify the proof for that case. The main idea of the method is inspired by \cite{Brown-Eastham-McCormack-98} and uses the discrete version
of a change of variables introduced by Eastham in \cite{Eastham-89}.

We could as well consider a finite sum of terms like \eqref{eq entries} as the potential,
\begin{equation*}
    b_n=\sum_{l=1}^L\frac{c_l\sin(2\omega_l n+\delta_l)}{n^{\gamma_l}}+q_n
\end{equation*}
with the same conditions imposed on $c_l,\omega_l,\gamma_l,\delta_l$ and $\{q_n\}_{n=1}^{\infty}$.
This would increase the number of critical points and complicate the notation and calculations, so
we restrict ourselves to the case of one such term only.

In the continuous case (differential Schr\"odinger operator on the half-line) Wigner-von Neumann potentials were studied in numerous works:
\cite{Matveev-73},\cite{Behncke-91-I},\cite{Behncke-91-II},\cite{Hinton-Klaus-Shaw-91},\cite{Klaus-91},\cite{Kurasov-92},\cite{Kurasov-96},\cite{Brown-Eastham-McCormack-98},\cite{Kurasov-Naboko-07},
and formulas for the spectral density analogous to \eqref{eq spectral density intro} were obtained in different variations in \cite{Matveev-73},\cite{Behncke-91-I} and \cite{Brown-Eastham-McCormack-98}.

 The second name author plans to use the Weyl-Titchmarsh type formula to study the spectral density of the Jacobi matrix $\mathcal J$. For the case $\gamma=1$ there are only two critical points: $\pm2\cos\omega$.
Analyzing the change of asymptotics of $P_n(\lambda)$ as $\lambda$ approaches the critical value $\lambda_0\in\{\pm2\cos\omega\}$
(remind that the type of asymptotics changes at $\lambda=\lambda_0$), we plan to prove in the forthcoming paper for the case $\gamma=1$ that
\begin{equation}\label{eq forthcoming}
    \rho'(\lambda)\sim\text{const}\cdot|\lambda-\lambda_0|^{\frac{|c|}{2|\sin\omega|}}\as\lambda\rightarrow\lambda_0,
\end{equation}
if $\{P_n(\lambda_0)\}_{n=1}^{\infty}$ is not a subordinate solution of the spectral equation.

Genralization to the case $\gamma<1$ is possible, but the problem seems to be much more difficult than for $\gamma=1$.

\section{Preliminaries}

For every complex $\lambda$ the spectral equation for $\mathcal J$
\begin{equation}\label{eq spectral equation prelim}
    u_{n-1}+b_nu_n+u_{n+1}=\lambda u_n,\ n\ge2
\end{equation}
has solutions $P_n(\lambda)$ (orthogonal polynomials of the first kind) and $Q_n(\lambda)$ (orthogonal polynomials of the second kind) such that
\begin{equation*}
    \begin{array}{l}
    P_1(\lambda)=1,\ P_2(\lambda)=\lambda-b_1
    \\
    Q_1(\lambda)=0,\ Q_2(\lambda)=1.
    \end{array}
\end{equation*}
For non-real values of $\lambda$ there exists $m(\lambda)$ (the Weyl function) such that
\begin{equation*}
    Q_n(\lambda)+m(\lambda)P_n(\lambda)
\end{equation*}
is \cite{Akhiezer-65} the unique (up to multiplication by a constant) solution of \eqref{eq spectral equation prelim}
that belongs to $l^2$. The Weyl function and the spectral density of $\mathcal J$ are related by the following equalities:
\begin{equation*}
    m(\lambda)=\int\limits_{\R}\frac{d\rho(x)}{x-\lambda},\
    \lambda\in\mathbb C\backslash\mathbb R,
\end{equation*}
\begin{equation*}
    \rho'(\lambda)=\frac1{\pi}\Im\, m(\lambda+i0)\text{ for a.a. }\lambda\in\mathbb R.
\end{equation*}
For every two solutions $u$ and $v$ of \eqref{eq spectral equation prelim} with the same parameter $\lambda$ the (discrete) Wronskian defined by
\begin{equation*}
    W(u,v):=u_nv_{n+1}-u_{n+1}v_n
\end{equation*}
is independent of $n$. For polynomials $P(\lambda)$ and $Q(\lambda)$ the Wronskian is equal to one for every $\lambda$.

\section{Reduction of the spectral equation to a
\\system of the L-diagonal form}

Consider the spectral equation for $\mathcal J$,
\begin{equation}\label{eq spectral equation}
    u_{n-1}+b_nu_n+u_{n+1}=\lambda u_n,\ n\ge2.
\end{equation}
Let us write it in the vector form,
\begin{equation}\label{eq spectral equation, vector form}
    \left(%
        \begin{array}{c}
        u_n \\
        u_{n+1} \\
        \end{array}%
    \right)
    =\left(%
        \begin{array}{cc}
        0 & 1 \\
        -1 & \lambda-b_n \\
        \end{array}%
    \right)
    \left(%
        \begin{array}{c}
        u_{n-1} \\
        u_n \\
        \end{array}%
    \right),
    \ n\ge 2.
\end{equation}
Consider a new parameter $z\in\overline{\mathbb D}$ (we denote by $\mathbb D$ the open unit disc) such that
\begin{equation*}
    \lambda=z+\frac1z
\end{equation*}
and conversely
\begin{equation*}
    z=\frac{\lambda-i\sqrt{4-\lambda^2}}2,
\end{equation*}
where the branch of the square root is chosen so that $z\in\mathbb D$ for $\lambda\in\C\backslash[-2;2]$, i.e., $z=-i$ for $\lambda=0$. Let
\begin{equation*}
    v_n:=\frac z{z^2-1}
    \left(%
    \begin{array}{cc}
    z & -1 \\
    -\frac1z & 1 \\
    \end{array}%
    \right)
    \left(%
    \begin{array}{c}
    u_n \\
    u_{n+1} \\
    \end{array}%
    \right),
\end{equation*}
which is equivalent to
\begin{equation*}
    \left(%
        \begin{array}{c}
        u_n \\
        u_{n+1} \\
        \end{array}%
    \right)
    =
    \left(%
        \begin{array}{cc}
        1 & 1 \\
        \frac1z & z\\
        \end{array}%
    \right)
    v_n,\ n\ge1.
\end{equation*}
We make this substitution to diagonalize the constant part of the coefficient matrix
\begin{equation*}
     \left(%
    \begin{array}{cc}
    0 & 1 \\
    -1 & \lambda\\
    \end{array}%
    \right)
    =
    \left(%
        \begin{array}{cc}
        1 & 1 \\
        \frac1z & z\\
        \end{array}%
    \right)
    \left(%
    \begin{array}{cc}
    \frac1z & 0 \\
    0 & z \\
    \end{array}%
    \right)
    \left(%
        \begin{array}{cc}
        1 & 1 \\
        \frac1z & z\\
        \end{array}%
    \right)
    ^{-1}.
\end{equation*}
Equation \eqref{eq spectral equation, vector form} becomes
\begin{equation}\label{eq system for v before transformation}
    v_{n+1}=\l[
    \left(%
        \begin{array}{cc}
        \frac1z & 0 \\
        0 & z \\
        \end{array}%
    \right)
    +\frac{b_{n+1}}{z^2-1}
    \left(%
        \begin{array}{cc}
        1 & z^2 \\
        -1 & -z^2 \\
        \end{array}%
    \right)
    \r]
    v_n,n\ge1.
\end{equation}
The goal of the present section is to reduce the system \eqref{eq system for v before transformation} to the "L-diagonal form". If we put
\begin{equation*}
    w_n:=T^{-1}_n(z)v_n,
\end{equation*}
then \eqref{eq system for v before transformation} becomes
\begin{equation*}
    w_{n+1}=T^{-1}_{n+1}(z)
    \l[
    \left(%
        \begin{array}{cc}
        \frac1z & 0 \\
        0 & z \\
        \end{array}%
    \right)
    +\frac{b_{n+1}}{z^2-1}
    \left(%
        \begin{array}{cc}
        1 & z^2 \\
        -1 & -z^2 \\
        \end{array}%
    \right)
    \r]
    T_n(z)w_n.
\end{equation*}
The system is in L-diagonal form if the coefficient matrix is a sum of diagonal and summable matrices.
So we have to find matrices $T_n(z)$ to provide this property. This is possible not for every $z\neq0,1,-1$.

Let us denote
\begin{equation}\label{eq U}
    U:=\C\backslash\{0,1,-1,e^{\pm i\omega},-e^{\pm i\omega},e^{\pm 2i\omega},-e^{\pm 2i\omega}\}.
\end{equation}

    \begin{lem}\label{lem T-n}
    Let $z\in U$. For every $n\in\mathbb N$ there exist matrices $R_n^{(2)}(z)$ and invertible matrices $T_n(z)$ such that
    \begin{equation*}
        T_n(z), T_n^{-1}(z)=I+o(1),
    \end{equation*}
    \begin{equation}\label{eq estimate}
        R_n^{(2)}(z)=O\l(\frac1{n^{3\gamma}}+|q_{n+1}|\r)
    \end{equation}
    as $\ninf$ and
    \begin{multline*}
        T_{n+1}^{-1}(z)
        \l[
        \left(%
            \begin{array}{cc}
            \frac1z & 0 \\
            0 & z \\
            \end{array}%
        \right)
        +\frac{b_{n+1}}{z^2-1}
        \left(%
            \begin{array}{cc}
            1 & z^2 \\
            -1 & -z^2 \\
            \end{array}%
        \right)
        \r]
        T_n(z)
        \\
        =
        \left(%
    \begin{array}{cc}
    \frac1z\l(1+\frac{\mu_2(z)}{n^{2\gamma}}\r) & 0 \\
    0 & z\l(1-\frac{\mu_2(z)}{n^{2\gamma}}\r) \\
    \end{array}%
    \right)
        +R_n^{(2)}(z),
    \end{multline*}
    where $\mu_2(z)$ is given by \eqref{eq mu-2}.
    $T_n(z)$ and $R_n^{(2)}(z)$ are also analytic in $U$ for all $n$, and on every compact subset of $U$ the estimate \eqref{eq estimate}
    is uniform with respect to $z$.
    \end{lem}

\begin{proof}
Let us denote
\begin{equation*}
    \begin{array}{l}
    \Lambda(z):=
    \left(%
        \begin{array}{cc}
        \frac1z & 0 \\
        0 & z \\
        \end{array}%
    \right),
    \\
    N_2(z):=
    \frac{ce^{i(2\omega+\delta)}}{2i(z^2-1)}
    \left(%
        \begin{array}{cc}
        1 & z^2 \\
        -1 & -z^2 \\
        \end{array}%
    \right),
    \\
    N_{-2}(z):=
    -\frac{ce^{-i(2\omega+\delta)}}{2i(z^2-1)}
    \left(%
        \begin{array}{cc}
        1 & z^2 \\
        -1 & -z^2 \\
        \end{array}%
    \right),
    \\
    R_n^{(0)}(z):=\frac{q_{n+1}}{z^2-1}
    \left(%
    \begin{array}{cc}
    1 & z^2 \\
    -1 & -z^2 \\
    \end{array}%
    \right),
    \end{array}
\end{equation*}
so that
\begin{equation*}
    \frac{b_{n+1}}{z^2-1}
    \left(%
    \begin{array}{cc}
    1 & z^2 \\
    -1 & -z^2 \\
    \end{array}%
    \right)
    =\frac{e^{2i\omega n}}{n^{\gamma}}N_2(z)+\frac{e^{-2i\omega n}}{n^{\gamma}}N_{-2}(z)+R_n^{(0)}(z).
\end{equation*}
We will find $T_n(z)$ in two steps.

At the first (and main) step of the construction of $T_n(z)$ let us find matrices $T_n^{(1)}(z)$ such that
\begin{multline}\label{eq main transformation}
    \l(T_{n+1}^{(1)}(z)\r)^{-1}
    \l[
    \left(%
        \begin{array}{cc}
        \frac1z & 0 \\
        0 & z \\
        \end{array}%
    \right)
    +\frac{b_{n+1}}{z^2-1}
    \left(%
        \begin{array}{cc}
        1 & z^2 \\
        -1 & -z^2 \\
        \end{array}%
    \right)
    \r]
    T_n^{(1)}(z)
    \\
    =
    \left(%
        \begin{array}{cc}
        \frac1z & 0 \\
        0 & z \\
        \end{array}%
    \right)
    +\frac{V(z)}{n^{2\gamma}}+R_n^{(1)}(z),
\end{multline}
where $V(z)$ does not depend on $n$ and $R_n^{(1)}(z)$ is summable.

Following the ideas of \cite{Eastham-89} and \cite{Brown-Eastham-McCormack-98}
we look for $T_n^{(1)}(z)$ of the form
\begin{multline}\label{eq T-1}
    T_n^{(1)}(z):=\exp
    \l(
    \frac{e^{2i\omega n}}{n^{\gamma}}X_2(z)+\frac{e^{-2i\omega n}}{n^{\gamma}}X_{-2}(z)
    \r.
    \\
    \l.
    +\frac{e^{4i\omega n}}{n^{2\gamma}}X_4(z)+\frac{e^{-4i\omega n}}{n^{2\gamma}}X_{-4}(z)
    \r),
\end{multline}
where $X_{\pm2}(z)$ and $X_{\pm4}(z)$ are to be determined. Define
\begin{multline}\label{eq M-pm}
    M_{\pm4}:=\frac12(\Lambda{X_{\pm2}}^2+e^{\pm4i\omega}{X_{\pm2}}^2\Lambda)+N_{\pm2}X_{\pm2}
    \\
    -e^{\pm2i\omega}X_{\pm2}N_{\pm2}-e^{\pm2i\omega}X_{\pm2}\Lambda X_{\pm2}
\end{multline}
and
\begin{multline}\label{eq V}
    V:=\frac12(\Lambda(X_2X_{-2}+X_{-2}X_2)+(X_2X_{-2}+X_{-2}X_2)\Lambda)+N_2X_{-2}
    \\
    +N_{-2}X_2-(e^{2i\omega}X_2N_{-2}+e^{-2i\omega}X_{-2}N_2)-(e^{2i\omega}X_2\Lambda X_{-2}+e^{-2i\omega}X_{-2}\Lambda X_2).
\end{multline}
Take expansions of $T_n^{(1)}$ and $(T_{n+1}^{(1)})^{-1}$ as $n\rightarrow\infty$ up to the terms of the order $\frac1{n^{2\gamma}}$.
After a long but transparent calculation we have:
\begin{multline}\label{eq long but transparent calculation}
    \l(T_{n+1}^{(1)}\r)^{-1}\l[\Lambda+\frac{e^{2i\omega n}}{n^{\gamma}}N_2+\frac{e^{-2i\omega n}}{n^{\gamma}}N_{-2}+R_n^{(0)}\r]T_n^{(1)}
    \\
    =
    \Lambda+\frac{e^{2i\omega n}}{n^{\gamma}}[N_2+\Lambda X_2-e^{2i\omega}X_2\Lambda]
    +\frac{e^{-2i\omega n}}{n^{\gamma}}[N_{-2}+\Lambda X_{-2}-e^{-2i\omega}X_{-2}\Lambda]
    \\
    +\frac{e^{4i\omega n}}{n^{2\gamma}}[M_4+\Lambda X_4-e^{4i\omega}X_4\Lambda]
    +\frac{e^{-4i\omega n}}{n^{2\gamma}}[M_{-4}+\Lambda X_{-4}-e^{-4i\omega}X_{-4}\Lambda]
    \\
    +\frac V{n^{2\gamma}}+O\l(\frac1{n^{3\gamma}}+|q_{n+1}|\r)\as\ninf,
\end{multline}
since $\|R_n^{(0)}\|=O(|q_{n+1}|)$.

We want to cancel the coefficients at $e^{\pm2i\omega n}$ and $e^{\pm4i\omega n}$ in \eqref{eq long but transparent calculation}
by suitable choice of $X_{\pm2}(z)$ and $X_{\pm4}(z)$, respectively. To this end, four conditions should be satisfied:
\begin{equation*}
    \begin{array}{l}
    e^{2i\omega}X_2\Lambda-\Lambda X_2=N_2,
    \\
    e^{-2i\omega}X_{-2}\Lambda-\Lambda X_{-2}=N_{-2},
    \\
    e^{4i\omega}X_4\Lambda-\Lambda X_4=M_4,
    \\
    e^{-4i\omega}X_{-4}\Lambda-\Lambda X_{-4}=M_{-4},
    \end{array}
\end{equation*}
We use the following lemma to solve them.

    \begin{lem}\label{lem commutator equation}
    If $\mu\neq1,z^2,\frac1{z^2}$, then the matrix
    \begin{equation*}
        X=
        \left(%
        \begin{array}{cc}
        \frac{zf_{11}}{\mu-1} & \frac{zf_{12}}{z^2\mu-1} \\
        \frac{zf_{21}}{\mu-z^2} & \frac{f_{22}}{z(\mu-1)} \\
        \end{array}%
        \right),
    \end{equation*}
    satisfies the equation
    \begin{equation*}
        \mu X\Lambda-\Lambda X=
        \left(%
        \begin{array}{cc}
        f_{11} & f_{12} \\
        f_{21} & f_{22} \\
        \end{array}%
        \right).
    \end{equation*}
    \end{lem}

\begin{proof}
The assertion can be verified by direct substitution.
\end{proof}

It follows that we can take
\begin{equation}\label{eq X-pm2}
    \begin{array}{l}
    X_2(z):=\frac{ce^{i(2\omega+\delta)}}{2i(z^2-1)}
    \left(%
        \begin{array}{cc}
        \frac z{e^{2i\omega}-1} & \frac{z^3}{z^2e^{2i\omega}-1} \\
        -\frac z{e^{2i\omega}-z^2} & -\frac z{e^{2i\omega}-1} \\
        \end{array}%
    \right),
    \\
    X_{-2}(z):=-\frac{ce^{-i(2\omega+\delta)}}{2i(z^2-1)}
    \left(%
        \begin{array}{cc}
        \frac z{e^{-2i\omega}-1} & \frac{z^3}{z^2e^{-2i\omega}-1} \\
        -\frac z{e^{-2i\omega}-z^2} & -\frac z{e^{-2i\omega}-1} \\
        \end{array}%
    \right)
    \end{array}
\end{equation}
and
\begin{equation*}
    X_{\pm4}(z):=
    \left(%
        \begin{array}{cc}
        \frac{z(M_{\pm4}(z))_{11}}{e^{\pm4i\omega}-1} & \frac{z(M_{\pm4}(z))_{12}}{z^2e^{\pm4i\omega}-1} \\
        \frac{z(M_{\pm4}(z))_{21}}{e^{\pm4i\omega}-z^2} & \frac{(M_{\pm4}(z))_{22}}{z(e^{\pm4i\omega}-1)} \\
        \end{array}%
    \right),
\end{equation*}
where $(M^{\pm}(z))_{11,12,21,22}$ are the entries of the matrix $M_{\pm4}(z)$, which are given by \eqref{eq M-pm} and \eqref{eq X-pm2}.
As we see, $X_{\pm2}(z)$ are defined and analytic in $\C\backslash\{1,-1,e^{\pm i\omega},-e^{\pm i\omega}\}$,
$M_{\pm4}(z)$ and $V(z)$ are defined and analytic in $\C\backslash\{0,1,-1,e^{\pm i\omega},-e^{\pm i\omega}\}$ and
$X_{\pm4}(z),R_n^{(2)}(z),T_n^{(1)}(z),(T_n^{(1)}(z))^{-1}$ are defined and analytic in $U$.

With this choice of $T_n^{(1)}(z)$  the remainder
\begin{multline*}
R_n^{(1)}(z):=\l(T_{n+1}^{(1)}(z)\r)^{-1}
    \l[
    \Lambda(z)+\frac{e^{2i\omega n}}{n^{\gamma}}N_2(z)
    \r.
    \\
    \l.
    +\frac{e^{-2i\omega n}}{n^{\gamma}}N_{-2}(z)+R_n^{(0)}(z)
    \r]
    T_n^{(1)}(z)-\Lambda(z)-\frac{V(z)}{n^{2\gamma}}
\end{multline*}
satisfies the estimate
\begin{equation*}
    \|R_n^{(1)}(z)\|=O\l(\frac1{n^{3\gamma}}+|q_{n+1}|\r)\as\ninf
\end{equation*}
uniformly with respect to $z$ on every compact subset of $U$.

At the second step of the construction of matrices $T_n(z)$ let us consider
\begin{equation*}
    T_n^{(2)}(z):=\exp\l(\frac{Y(z)}{n^{2\gamma}}\r)
\end{equation*}
with some $Y(z)$, which is to be determined. Taking the expansions of $T_n^{(2)}$ and $(T_{n+1}^{(2)})^{-1}$ up to the order of $\frac1{n^{2\gamma}}$, we have:
\begin{multline*}
    \l(T_{n+1}^{(2)}\r)^{-1}\l(\Lambda+\frac V{n^{2\gamma}}+R_n^{(1)}\r)T_n^{(2)}
    \\
    =\Lambda+\frac1{n^{2\gamma}}(V-[Y,\Lambda])+O\l(\frac1{n^{3\gamma}}+|q_{n+1}|\r)\as\ninf.
\end{multline*}
Let us cancel the anti-diagonal entries of $V-[Y,\Lambda]$ by the choice of $Y$. This leads to the equation
\begin{equation*}
    [Y,\Lambda]
    =
    \left(%
        \begin{array}{cc}
        0 & V_{12} \\
        V_{21} & 0 \\
        \end{array}%
    \right).
\end{equation*}
We can take
\begin{equation*}
    Y(z):=\frac z{z^2-1}
    \left(%
        \begin{array}{cc}
        0 & V_{12}(z) \\
        -V_{21}(z) & 0 \\
        \end{array}%
    \right).
\end{equation*}
What rests is
\begin{equation*}
    \text{diag }V(z)=\mu_2(z)
    \left(%
    \begin{array}{cc}
    \frac1z & 0 \\
    0 & z \\
    \end{array}%
    \right),
\end{equation*}
which can be seen from \eqref{eq V} and \eqref{eq X-pm2} by a straightforward calculation. Matrices $Y(z),T_n^{(2)}(z)$ and $(T_n^{(2)}(z))^{-1}$
are defined and analytic in $U$. The remainder in the system after the transformation,
\begin{multline*}
    R_n^{(2)}(z):=
    \l(T_{n+1}^{(2)}(z)\r)^{-1}\l(\Lambda(z)+\frac{V(z)}{n^{2\gamma}}+R_n^{(1)}(z)\r)T_n^{(2)}(z)
    \\
    -\Lambda(z)-\frac{\text{diag}V(z)}{n^{2\gamma}},
\end{multline*}
satisfies the estimate
\begin{equation*}
    \|R_n^{(2)}(z)\|=O\l(\frac1{n^{3\gamma}}+|q_{n+1}|\r)\as\ninf
\end{equation*}
uniformly with respect to $z$ on every compact subset of $U$.
Taking
\begin{equation*}
    T_n(z):=T_n^{(1)}(z)T_n^{(2)}(z)
\end{equation*}
we complete the proof.
\end{proof}

Finally, we have come to the following system of the L-diagonal form:
\begin{equation}\label{eq system for w}
    w_{n+1}=\l[
    \left(%
        \begin{array}{cc}
        \frac1z\l(1+\frac{\mu_2(z)}{n^{2\gamma}}\r) & 0
        \\
        0 & z\l(1-\frac{\mu_2(z)}{n^{2\gamma}}\r)
        \end{array}%
    \right)
    +R_n^{(2)}(z)\r]w_n.
\end{equation}
It is easy to check that for $z\in\mathbb T\cap U$ the value $\mu_2(z)$ is pure imaginary.

\section{Asymptotic results}

In this section we prove several results needed for the analysis of the system \eqref{eq system for w}. They are more or less standard, and the approach
is similar to \cite{Janas-Moszynski-03}, \cite{Silva-04} and \cite{Benzaid-Lutz-87}.
In the cited papers, the existence of a base of solutions with special asymptotic behavior is proven. Here we find the asymptotic of (roughly speaking) generic solution
defined by its initial value.

Let us use the following notation.
\begin{equation*}
    \begin{array}{c}
    \sum\limits_{n=1}^0:=0,\ \prod\limits_{n=1}^0:=I,
    \\
    \text{ and for } n_1,n_2\ge1,
    \\
    \sum\limits_{n=n_1}^{n_2}:=\sum\limits_{n=1}^{n_2}-\sum\limits_{n=1}^{n_1-1},
    \\
    \prod\limits_{n=n_1}^{n_2}:=\prod\limits_{n=1}^{n_2}\l(\prod\limits_{n=1}^{n_1-1}\r)^{-1}.
    \end{array}
\end{equation*}
The first lemma is a kind of discrete variation of parameters.

    \begin{lem}\label{lem variation}
    Let $f\in\C^2$ and let the  matrices $\Lambda_n$ be invertible for every $n\ge1$.
    If for every $n\ge1$
    \begin{equation}\label{eq variation of parameters}
        x_n=\l(\prod_{l=1}^{n-1}\Lambda_l\r)f+\sum_{k=1}^{n-1}\l(\prod_{l=k+1}^{n-1}\Lambda_l\r)R_kx_k,
    \end{equation}
    then
    \begin{equation}\label{eq system most general}
        x_{n+1}=(\Lambda_n+R_n)x_n,
    \end{equation}
    for every $n\ge1 .$
    \end{lem}

\begin{proof}
Consider
\begin{equation*}
    y_n^{(1)}:=\l(\prod_{l=1}^{n-1}\Lambda_l\r)^{-1}x_n.
\end{equation*}
We can rewrite \eqref{eq system most general} as
\begin{equation}\label{eq no 1}
    y_{n+1}^{(1)}-y_n^{(1)}=\l(\prod_{l=1}^n\Lambda_l\r)^{-1}R_nx_n.
\end{equation}
At the same time \eqref{eq variation of parameters} is equivalent to
\begin{equation}\label{eq equation for y}
    y_n^{(1)}=f+\sum_{k=1}^{n-1}\l(\prod_{l=1}^k\Lambda_l\r)^{-1}R_kx_k.
\end{equation}
Clearly, \eqref{eq no 1} follows from \eqref{eq equation for y}.
\end{proof}

    \begin{rem}
    Lemma \ref{lem variation} says that $x$ given by \eqref{eq variation of parameters} is the
    solution of the system \eqref{eq system most general} with $x_1=f$.
    Every solution $x$ of \eqref{eq system most general} can be represented in the form
    \eqref{eq variation of parameters} with $f:=x_1$.
    \end{rem}

In what follows let us consider systems with matrices $\Lambda_n$ of the form
\begin{equation*}
    \Lambda_n:=
    \left(%
        \begin{array}{cc}
        \lambda_n & 0 \\
        0 & \frac1{\lambda_n} \\
        \end{array}%
    \right),
\end{equation*}
where $\lambda_n\in\C$. The following lemma gives an estimate on growth of solutions
of the system \eqref{eq system most general}.

    \begin{lem}\label{lem estimate of solution growth}
    Let
    \begin{equation*}
        \sum\limits_{k=1}^{\infty}\frac{\|R_k\|}{|\lambda_k|}<\infty
    \end{equation*}
    and let there exist $M$ such that for every $m\ge n$
    \begin{equation*}
        \prod\limits_{l=n+1}^m|\lambda_l|\ge\frac1M.
    \end{equation*}
    Every solution $x$ of the system
    \begin{equation}\label{eq system in Levinson form}
        x_{n+1}=\l[
        \left(%
        \begin{array}{cc}
        \lambda_n & 0 \\
        0 & \frac1{\lambda_n} \\
        \end{array}%
        \right)
        +R_n\r]x_n,\ n\ge1
    \end{equation}
    satisfies the following estimate:
    \begin{equation}\label{eq estimate on growth}
        \|x_n\|\le\l(\prod_{l=1}^{n-1}|\lambda_l|\r)
        \exp\l((1+M^2)\sum\limits_{k=1}^{\infty}\frac{\|R_k\|}{|\lambda_k|}\r)(1+M^2)\|x_1\|.
    \end{equation}
    \end{lem}

\begin{proof}
Let
\begin{equation*}
    f=
    \left(%
        \begin{array}{c}
        f_1 \\
        f_2 \\
        \end{array}%
    \right)
    :=x_1.
\end{equation*}
Define
\begin{equation*}
    y_n^{(2)}:=\frac{x_n}{\prod\limits_{l=1}^{n-1}\lambda_l}.
\end{equation*}
Then $\{y_n^{(2)}\}$ satisfies the equation (by using Lemma \ref{lem variation})
\begin{equation}\label{eq Volterra equation}
    y_n^{(2)}=
    \left(%
        \begin{array}{cc}
        1 & 0 \\
        0 & \prod\limits_{l=1}^{n-1}\frac1{\lambda_l^2} \\
        \end{array}%
    \right)
    f+\sum_{k=1}^{n-1}
    \left(%
        \begin{array}{cc}
        1 & 0 \\
        0 & \prod\limits_{l=k+1}^{n-1}\frac1{\lambda_l^2} \\
        \end{array}%
    \right)
    \frac{R_k}{\lambda_k}y_k^{(2)}.
\end{equation}
Consider this as an equation in the Banach space $l^{\infty}(\C^2)$.
Denote
\begin{equation*}
    \begin{array}{l}
        \text{the vector }\hat f:=\l\{
        \left(%
            \begin{array}{c}
            f_1 \\
            \l(\prod\limits_{l=1}^{n-1}\frac1{\lambda_l^2}\r)f_2 \\
            \end{array}%
        \right)
        \r\}_{n=1}^{\infty},
        \\
        \text{and the operator }V:\{u_n\}_{n=1}^{\infty}\mapsto
        \l\{\sum\limits_{k=1}^{n-1}
        \left(%
            \begin{array}{cc}
            1 & 0 \\
            0 & \prod\limits_{l=k+1}^{n-1}\frac1{\lambda_l^2} \\
            \end{array}%
            \right)
        \frac{R_k}{\lambda_k}u_k\r\}_{n=1}^{\infty}.
    \end{array}
\end{equation*}
Equation \eqref{eq Volterra equation} reads in this notation:
\begin{equation*}
    y=\hat f+Vy.
\end{equation*}
The powers of the operator $V$ can be estimated as follows:
\begin{equation*}
    \|V^m\|_{l^{\infty}}
    \le
    \frac{\l((1+M^2)\sum\limits_{k=1}^{\infty}\frac{\|R_k\|}{|\lambda_k|}\r)^m}{m!},
\end{equation*}
and so $(I-V)^{-1}$ exists and
\begin{equation*}
    \|(I-V)^{-1}\|_{l^{\infty}}
    \le
    \exp\l((1+M^2)\sum\limits_{k=1}^{\infty}\frac{\|R_k\|}{|\lambda_k|}\r).
\end{equation*}
Therefore
\begin{equation*}
    \|y^{(2)}\|_{l^{\infty}}
    \le
    \exp\l((1+M^2)\sum\limits_{k=1}^{\infty}\frac{\|R_k\|}{|\lambda_k|}\r)(1+M^2)\|f\|.
\end{equation*}
Returning to the solution $x$ we have arrived at the desired estimate \eqref{eq estimate on growth}
\end{proof}

The following lemma gives asymtotics of the solutions of the system \eqref{eq system in Levinson form}.

    \begin{lem}\label{lem asymptotics}
    Let
    \begin{equation*}
        \sum\limits_{k=1}^{\infty}\frac{\|R_k\|}{|\lambda_k|}<\infty
    \end{equation*}
    and let there exist $M$ such that for every $m\ge n$,
    \begin{equation}\label{eq Levinson condition}
        \prod\limits_{l=n+1}^m|\lambda_l|\ge\frac1M.
    \end{equation}
    Suppose that $x$ is a solution of the system
    \begin{equation}\label{eq system in Levinson form 2 time}
        x_{n+1}=\l[
    \left(%
        \begin{array}{cc}
        \lambda_n & 0 \\
        0 & \frac1{\lambda_n} \\
        \end{array}%
    \right)
    +R_n\r]x_n,\ n\ge1.
    \end{equation}
    \\
    a) If
    \begin{equation}\label{eq elliptic condition}
        \prod_{l=1}^{\infty}|\lambda_l|<\infty,
    \end{equation}
    then
    \begin{equation*}
        \lim\limits_{\ninf}\l(\prod\limits_{l=1}^{n-1}\Lambda_l\r)^{-1}x_n
        =x_1+\sum\limits_{k=1}^{\infty}\l(\prod\limits_{l=1}^k\Lambda_l\r)^{-1}R_kx_k
    \end{equation*}
    (the limit and the sum both exist and the equality holds).
    \\
    b) If
    \begin{equation}\label{eq hyperbolic condition}
        \prod_{l=1}^{\infty}\lambda_l=\infty,
    \end{equation}
    then
    \begin{equation*}
        \lim\limits_{\ninf}
        \frac{x_n}{\prod\limits_{l=1}^{n-1}\lambda_l}
        =
        \left(%
            \begin{array}{cc}
            1 & 0 \\
            0 & 0 \\
            \end{array}%
        \right)
        \l[
        x_1+\sum\limits_{k=1}^{\infty}
        \frac{R_kx_k}{\prod\limits_{l=1}^k\lambda_l}\r]
    \end{equation*}
    (the limit and the sum both exist and the equality holds).
    \end{lem}

\begin{proof}
Case a). Equation \eqref{eq equation for y} can be rewritten as follows:
\begin{equation}\label{eq no 2}
     \l(\prod\limits_{l=1}^{n-1}\Lambda_l\r)^{-1}x_n
     =x_1+\sum\limits_{k=1}^{n-1}\l(\prod\limits_{l=1}^k\Lambda_l\r)^{-1}R_kx_k.
\end{equation}
Let us show that the sum on the right-hand side is convergent.
By Lemma \ref{lem estimate of solution growth},
\begin{multline*}
    \l\|\l(\prod\limits_{l=1}^k\Lambda_l\r)^{-1}R_kx_k\r\|
    \le\frac{\|R_k\|}{|\lambda_k|}
    \l\|\l(\prod\limits_{l=1}^k\frac{\Lambda_l}{\lambda_l}\r)^{-1}\r\|
    \\
    \times
    \exp\l((1+M^2)\sum\limits_{m=1}^{\infty}\frac{\|R_m\|}{|\lambda_m|}\r)(1+M^2)\|x_1\|.
\end{multline*}
Since
\begin{equation*}
    \l\|\l(\prod\limits_{l=1}^k\frac{\Lambda_l}{\lambda_l}\r)^{-1}\r\|
    =
    \l\|
    \left(%
        \begin{array}{cc}
        1 & 0 \\
        0 & \prod\limits_{l=1}^k\lambda_l^2 \\
        \end{array}%
    \right)
    \r\|
    \le{\sqrt{1+\l(\prod_{l=1}^k|\lambda_l|\r)^2}}
\end{equation*}
is bounded by hypothesis, we have:
\begin{equation*}
    \l\|\l(\prod\limits_{l=1}^k\Lambda_l\r)^{-1}R_kx_k\r\|
    \le
    \text{const}\cdot\frac{\|R_k\|}{|\lambda_k|}
\end{equation*}
which is summable. Therefore the limit in \eqref{eq no 2} as $n\rightarrow\infty$ exists.

Case b). Consider the sum on the right-hand side of \eqref{eq Volterra equation}.
Lemma \ref{lem estimate of solution growth} yields:
\begin{multline*}
    \l\|
    \left(%
        \begin{array}{cc}
        1 & 0 \\
        0 & \prod\limits_{l=k+1}^{n-1}\frac1{\lambda_l^2} \\
        \end{array}%
    \right)
    \frac{R_k}{\lambda_k}y_k^{(2)}\r\|
    \le
    (1+M^2)\frac{\|R_k\|}{|\lambda_k|}\|y_k^{(2)}\|
    \\
    \le
    \frac{\|R_k\|}{|\lambda_k|}(1+M^2)^2
    \exp\l((1+M^2)\sum\limits_{m=1}^{\infty}\frac{\|R_m\|}{|\lambda_m|}\r)\|x_1\|,
\end{multline*}
which is summable. Since \eqref{eq hyperbolic condition} holds, by the Lebesgue's dominated convergence theorem
there exists the limit as $\ninf$ in \eqref{eq Volterra equation}:
\begin{equation*}
    \lim\limits_{\ninf}y_n^{(2)}=
    \left(%
        \begin{array}{cc}
        1 & 0 \\
        0 & 0 \\
        \end{array}%
    \right)
    x_1+\sum_{k=1}^{\infty}
    \left(%
        \begin{array}{cc}
        1 & 0 \\
        0 & 0 \\
        \end{array}%
    \right)\frac{R_k}{\lambda_k}y_k^{(2)}.
\end{equation*}
Returning to $x$ from $y^{(2)}$ we obtain the assertion of the lemma in the case b).
\end{proof}

    \begin{rem}
    Condition \eqref{eq Levinson condition} together with \eqref{eq elliptic condition}
    or \eqref{eq hyperbolic condition} is a case of the standard dichotomy (Levinson) condition,
    cf. \cite{Benzaid-Lutz-87}, \cite{Janas-Moszynski-03}, \cite{Silva-04}, \cite{Silva-07}.
    \end{rem}

\section{Asymptotics of polynomials, Jost function and the spectral density}
In this section we apply the results of the previous two sections to find asymptotics
of polynomials $P_n(\lambda)$ associated to the matrix $\mathcal J$ and to prove the Weyl-Titchmarsh type formula for the spectral density.

    \begin{thm}\label{thm asymptotics of polynomials}
    Let $\{b_n\}$ be given by \eqref{eq entries} and the condition \eqref{eq conditions} hold.
    Let $P_n(\lambda)$ be orthonormal polynomials associated to the Jacobi matrix $\mathcal J$ given by \eqref{eq definition of the operator}.
    Then for every $z\in\overline{\mathbb D}\cap U$ (where $U$ is given by \eqref{eq U}) there exists $F(z)$ (the Jost function) such that:
    \begin{itemize}
    \item if $z\in\mathbb D\backslash\{0\}$ (i.e., $\lambda=z+\frac1z\in\C\backslash[-2;2]$), then
    \begin{multline}\label{eq asymptotics of P hyperbolic}
    P_n\l(z+\frac1z\r)=\frac{zF(z)}{1-z^2}\frac1{z^n}\exp\l(\frac{\mu_2(z)n^{1-2\gamma}}{1-2\gamma}\r)
    \\
    +o\l(\frac1{|z|^n}\exp\l(\frac{\Re\mu_2(z)n^{1-2\gamma}}{1-2\gamma}\r)\r)\as\ninf
    \end{multline}
    (where $\mu_2(z)$ is given by \eqref{eq mu-2}),
    \item if
    $z\in\mathbb T\cap U$
    (i.e., $\lambda=z+\frac1z\in(-2;2)\backslash\{\pm2\cos\omega,\pm2\cos2\omega\}$),
    then
    \begin{multline}\label{eq asymptotics of P elliptic}
    P_n\l(z+\frac1z\r)=\frac{zF(z)}{1-z^2}\frac1{z^n}\exp\l(\frac{\mu_2(z)n^{1-2\gamma}}{1-2\gamma}\r)
    \\
    +\frac{z\overline{F(z)}}{z^2-1}z^n\exp\l(-\frac{\mu_2(z)n^{1-2\gamma}}{1-2\gamma}\r)
    +o(1)\as\ninf.
    \end{multline}
    \end{itemize}
    Function $F$ is analytic in $\mathbb D\backslash\{0\}$ and continuous in $\mathbb T\cap U$.
    \end{thm}

\begin{proof}
For $z\in U$ every solution $u$ of the spectral equation \eqref{eq spectral equation} corresponds to the solution $w$ of \eqref{eq system for w} by the equality
\begin{equation*}
    \left(%
        \begin{array}{c}
        u_n \\
        u_{n+1} \\
        \end{array}%
    \right)
    =
    \left(%
        \begin{array}{cc}
        1 & 1 \\
        \frac1z & z \\
        \end{array}%
    \right)
    T_n(z)w_n.
\end{equation*}
Let us define the solution $\varphi(z)$ of \eqref{eq system for w} that corresponds to polynomials $P_n(\lambda)$:
\begin{equation}\label{eq relation between phi and P}
    \varphi_n(z):=T_n^{-1}(z)
    \left(%
        \begin{array}{cc}
        1 & 1 \\
        \frac1z & z \\
        \end{array}%
    \right)
    ^{-1}
    \left(%
        \begin{array}{c}
        P_n\l(z+\frac1z\r) \\
        P_{n+1}\l(z+\frac1z\r) \\
        \end{array}%
    \right).
\end{equation}
Define
\begin{equation}\label{eq lambda}
    \lambda_n(z):=\l\{
        \begin{array}{l}
        \frac1z\exp\l(\frac{\mu_2(z)}{1-2\gamma}((n+1)^{1-2\gamma}-n^{1-2\gamma})\r),\text{ if }n\ge2,
        \\
        \frac1{z^2}\exp\l(\frac{\mu_2(z)}{1-2\gamma}2^{1-2\gamma}\r),\text{ if }n=1
        \end{array}
    \r.
\end{equation}
and
\begin{equation*}
R_n^{(3)}(z):=R_n^{(2)}(z)+
\left(%
\begin{array}{cc}
  \frac1z\l(1+\frac{\mu_2(z)}{n^{2\gamma}}\r)-\lambda_n(z) & 0 \\
  0 & z\l(1-\frac{\mu_2(z)}{n^{2\gamma}}\r)-\frac1{\lambda_n(z)} \\
\end{array}%
\right),
\end{equation*}
so that the system \eqref{eq system for w} reads:
\begin{equation}\label{eq systam for w 2 time}
    w_{n+1}=\l[
    \left(%
        \begin{array}{cc}
        \lambda_n(z) & 0 \\
        0 & \frac1{\lambda_n(z)} \\
        \end{array}%
    \right)
    +R_n^{(3)}(z)\r]w_n,\ n\ge2.
\end{equation}
Let us check that Lemmas \ref{lem estimate of solution growth} and \ref{lem asymptotics} are applicable to this system. With the definition \eqref{eq lambda},
the product of diagonal entries looks simple:
\begin{equation*}
    \prod\limits_{l=1}^{n-1}\lambda_l(z)=\frac1{z^n}\exp\l(\frac{\mu_2(z)n^{1-2\gamma}}{1-2\gamma}\r),\ n\ge2.
\end{equation*}
Let $K$ be a compact subset of $\overline{\mathbb D}\cap U$. We have for $m\ge n$:
\begin{multline*}
    \prod\limits_{l=n+1}^m|\lambda_l(z)|=\frac1{|z|^{m-n}}\exp\l(\frac{\Re\mu_2(z)}{1-2\gamma}((m+1)^{1-2\gamma}-(n+1)^{1-2\gamma})\r)
    \\
    \ge
    \frac1{|z|^{m-n}}\exp\l(-\frac{|\Re\mu_2(z)|}{1-2\gamma}(m-n)^{1-2\gamma}\r),
\end{multline*}
where we used the inequality
\begin{equation*}
    (m+1)^{1-2\gamma}-(n+1)^{1-2\gamma}\le(m-n)^{1-2\gamma}
\end{equation*}
(which holds because the function $x\mapsto x^{1-2\gamma}$ is concave). Further, since $\Re\mu_2(z)=0$ for $z\in\mathbb T$, the function
\begin{equation*}
    z\mapsto\frac{\Re\mu_2(z)}{1-|z|}
\end{equation*}
is smooth on $U$, hence there exists $c_1(K)$ such that for every $z\in K$
\begin{equation*}
    |\Re\mu_2(z)|\le c_1(K)(1-|z|).
\end{equation*}
Also for every $z$
\begin{equation*}
    |z|\le e^{|z|-1},
\end{equation*}
therefore
\begin{equation*}
    \prod\limits_{l=n+1}^m|\lambda_l(z)|\ge\exp\l[(1-|z|)\l(m-n-\frac{c_1(K)}{1-2\gamma}(m-n)^{1-2\gamma}\r)\r].
\end{equation*}
Let
\begin{equation*}
    c_2(K):=\sup_{x\ge0}\l(\frac{c_1(K)}{1-2\gamma}x^{1-2\gamma}-x\r),
\end{equation*}
which is finite. We have: for every $z\in K$ and $m\ge n$,
\begin{equation*}
    (1-|z|)\l(m-n-\frac{c_1(K)}{1-2\gamma}(m-n)^{1-2\gamma}\r)\in[-c_2(K);+\infty)
\end{equation*}
and
\begin{equation*}
    \prod\limits_{l=n+1}^m|\lambda_l(z)|\ge e^{-c_2(K)}.
\end{equation*}
Further,
\begin{equation*}
    \|R_n^{(2)}(z)\|=O\l(\frac1{n^{3\gamma}}+|q_{n+1}|\r)\as\ninf
\end{equation*}
uniformly with respect to $z\in K$ and
\begin{multline*}
    \frac1z\l(1+\frac{\mu_2(z)}{n^{2\gamma}}\r)-\lambda_n(z)
    \\
    =\frac1z\l(1+\frac{\mu_2(z)}{n^{2\gamma}}-\exp\l[\frac{\mu_2(z)}{1-2\gamma}\l((n+1)^{1-2\gamma}-n^{1-2\gamma}\r)\r]\r)
    \\
    =O\l(\frac1{n^{2\gamma+1}}\r)\as\ninf
\end{multline*}
Analogously
\begin{equation*}
    z\l(1-\frac{\mu_2(z)}{n^{2\gamma}}\r)-\frac1{\lambda_n(z)}=O\l(\frac1{n^{2\gamma+1}}\r),
\end{equation*}
and finally
\begin{equation*}
    \|R_n^{(3)}(z)\|=O\l(\frac1{n^{3\gamma}}+|q_{n+1}|\r)\as\ninf
\end{equation*}
uniformly with respect to $z\in K$. Also for every $z\in K$ and $n$,
\begin{equation*}
    \frac1{|\lambda_n(z)|}\le|z|\exp\l(\frac{2|\Re\mu_2(z)|}{1-2\gamma}\r).
\end{equation*}
Thus the sum
\begin{equation*}
    \sum_{n=1}^{\infty}\frac{\|R_n^{(3)}(z)\|}{|\lambda_n(z)|}
\end{equation*}
as a function of $z$ is bounded on $K$ (in fact it is continuous in $U$). For $z\in\mathbb T\cap K$
\begin{equation*}
    \prod\limits_{l=1}^n\lambda_l(z)
\end{equation*}
is bounded, while for $z\in\mathbb D\cap K$
\begin{equation*}
    \prod\limits_{l=1}^n\lambda_l(z)\rightarrow\infty\as\ninf.
\end{equation*}
We see that Lemma \ref{lem asymptotics} is applicable. It yields:
for every $z\in K$ there exists
\begin{equation}\label{eq Phi}
    \Phi(z):=\l(\varphi_1(z)+\sum\limits_{n=1}^{\infty}z^n\exp\l(-\frac{\mu_2(z)n^{1-2\gamma}}{1-2\gamma}\r)R_n^{(3)}(z)\varphi_n(z)\r)_1,
\end{equation}
and for every $z\in\mathbb T\cap K$ there exists
\begin{equation*}
    \widetilde\Phi(z):=\l(\varphi_1(z)+\sum\limits_{n=1}^{\infty}\frac1{z^n}\exp\l(\frac{\mu_2(z)n^{1-2\gamma}}{1-2\gamma}\r)R_n^{(3)}(z)\varphi_n(z)\r)_2,
\end{equation*}
and solution $\varphi(z)$ has the following asymptotics. For $z\in\mathbb T\cap K$ (case (a) of Lemma \ref{lem asymptotics}),
\begin{equation}\label{eq phi asymptotics elliptic case}
    \varphi_n(z)=
    \left(%
    \begin{array}{l}
    \frac1{z^n}\exp\l(\frac{\mu_2(z)n^{1-2\gamma}}{1-2\gamma}\r) \ \ \ \ \ \ \ \ \ \ 0 \\
    \ \ \ \ \ \ \ \ \ 0 \ \ \ \ \ \ \ \ \ \  z^n\exp\l(-\frac{\mu_2(z)n^{1-2\gamma}}{1-2\gamma}\r) \\
    \end{array}%
    \right)
    \l(
    \left(%
        \begin{array}{c}
        \Phi(z) \\
        \widetilde \Phi(z) \\
        \end{array}%
    \right)
    +o(1)
    \r)
\end{equation}
as $\ninf$ and for $z\in\mathbb D\cap K$ (case (b) of Lemma \ref{lem asymptotics}),
\begin{equation}\label{eq phi asymptotics hyperbolic case}
    \varphi_n(z)=\frac1{z^n}\exp\l(\frac{\mu_2(z)n^{1-2\gamma}}{1-2\gamma}\r)
    \l(
    \left(%
        \begin{array}{c}
        \Phi(z) \\
        0 \\
        \end{array}%
    \right)
    +o(1)
    \r)
    \as\ninf.
\end{equation}
As we see, Lemma \ref{lem estimate of solution growth} is also applicable. Let
\begin{equation*}
    c_3(K):=\max_{z\in K}\sum_{n=1}^{\infty}\frac{\|R_n^{(3)}(z)\|}{|\lambda_n(z)|},
\end{equation*}
\begin{equation*}
    c_4(K):=\exp\l(\l(1+e^{2c_2(K)}\r)c_3(K)\r)\l(1+e^{2c_2(K)}\r)\max_{z\in K}\|\varphi_1(z)\|.
\end{equation*}
Lemma \ref{lem estimate of solution growth} yields: for every $z\in K$ and $n$
\begin{equation*}
    \|\varphi_n(z)\|\le\frac{c_4(K)}{|z|^n}\exp\l(\frac{\Re\mu_2(z)n^{1-2\gamma}}{1-2\gamma}\r).
\end{equation*}
Consider the expression for $\Phi(z)$, \eqref{eq Phi}. We have:
\begin{multline*}
    \l\|z^n\exp\l(-\frac{\mu_2(z)n^{1-2\gamma}}{1-2\gamma}\r)R_n^{(3)}(z)\varphi_n(z)\r\|\le c_4(K)\|R_n^{(3)}(z)\|
    \\
    =O\l(\frac1{n^{3\gamma}}+|q_{n+1}|\r)\as\ninf
\end{multline*}
uniformly with respect to $z\in K$. It follows that the function $\Phi$ is analytic in the interior of $K$ and continuous in $K$.
Since the set $K\subset\overline{\mathbb D}\cap U$ is arbitrary, $\Phi$ exists and is continuous in $\overline{\mathbb D}\cap U$
and is analytic in $\mathbb D\backslash\{0\}$. Asymptotics \eqref{eq phi asymptotics elliptic case} holds for every $z\in\mathbb T\cap U$
and asymptotics \eqref{eq phi asymptotics hyperbolic case} holds for every $z\in\mathbb D\backslash\{0\}$.

If follows that for $z\in\mathbb D\backslash\{0\}$,
\begin{multline*}
    \left(%
        \begin{array}{c}
        P_n\l(z+\frac1z\r) \\
        P_{n+1}\l(z+\frac1z\r) \\
        \end{array}%
    \right)
    =\frac1{z^n}\exp\l(\frac{\mu_2(z)n^{1-2\gamma}}{1-2\gamma}\r)
    \left(%
        \begin{array}{cc}
        1 & 1 \\
        \frac1z & z \\
        \end{array}%
    \right)
    T_n(z)
    \\
    \times\l(
    \left(%
        \begin{array}{c}
        \Phi(z) \\
        0 \\
        \end{array}%
    \right)
    +o(1)\r)
    =\Phi(z)\frac1{z^n}\exp\l(\frac{\mu_2(z)n^{1-2\gamma}}{1-2\gamma}\r)
    \l(
    \left(%
        \begin{array}{c}
        1 \\
        \frac1z \\
        \end{array}%
    \right)
    +o(1)
    \r).
\end{multline*}
as $\ninf$, since $T_n(z)=I+o(1)$. If we define
\begin{equation}\label{eq F}
    F(z):=\Phi(z)\frac{1-z^2}z,
\end{equation}
then we arrive at the first assertion of the theorem.

It also follows in an analogous fashion that for $z\in\mathbb T\cap U$,
\begin{multline}\label{eq asymptotics of P incomplete}
    \left(%
        \begin{array}{c}
        P_n\l(z+\frac1z\r) \\
        P_{n+1}\l(z+\frac1z\r) \\
        \end{array}%
    \right)
    =\Phi(z)\frac1{z^n}\exp\l(\frac{\mu_2(z)n^{1-2\gamma}}{1-2\gamma}\r)
    \left(%
        \begin{array}{c}
        1 \\
        \frac1z \\
        \end{array}%
    \right)
    \\
    +\widetilde\Phi(z)z^n\exp\l(-\frac{\mu_2(z)n^{1-2\gamma}}{1-2\gamma}\r)
    \left(%
        \begin{array}{c}
        1 \\
        z \\
        \end{array}%
    \right)
    +o(1)\as\ninf.
\end{multline}
The first component of this vector equality describes the asymptotic of $P_n\l(z+\frac1z\r)$, and to complete the proof we need only the following lemma.

    \begin{lem}\label{lem conjugation}
    For every $z\in\mathbb T\cap U$,
    \begin{equation*}
        \widetilde\Phi(z)=\overline{\Phi(z)}.
    \end{equation*}
    \end{lem}

\begin{proof}
This follows from the fact that values of polynomials $ P_n\l(z+\frac1z\r)$ for $z\in\mathbb T$ are real. Consider the imaginary part of the first component of \eqref{eq asymptotics of P incomplete}:
\begin{multline*}
    0=\frac{\Phi(z)-\overline{\widetilde\Phi(z)}}{2i}\frac1{z^n}\exp\l(\frac{\mu_2(z)n^{1-2\gamma}}{1-2\gamma}\r)
    \\
    +\frac{\widetilde\Phi(z)-\overline{\Phi(z)}}{2i}z^n\exp\l(-\frac{\mu_2(z)n^{1-2\gamma}}{1-2\gamma}\r)
    +o(1)\as\ninf.
\end{multline*}
Suppose that
\begin{equation*}
    \widetilde\Phi(z)\neq\overline{\Phi(z)}.
\end{equation*}
Then
\begin{equation*}
    z^{2n}\exp\l(2i\arg\l(\overline{\Phi(z)}-\widetilde\Phi(z)\r)-\frac{2\mu_2(z)}{1-2\gamma}n^{1-2\gamma}\r)\rightarrow1\as\ninf.
\end{equation*}
Let
\begin{equation*}
    \hat\Phi(z):=\exp\l[2i\arg\l(\overline{\Phi(z)}-\widetilde\Phi(z)\r)\r],
\end{equation*}
\begin{equation*}
    \hat\mu(z):=-\frac{2\mu_2(z)}{1-2\gamma}.
\end{equation*}
We have:
\begin{equation*}
    \hat\Phi z^{2n}e^{\hat\mu n^{1-2\gamma}}\rightarrow1\as\ninf.
\end{equation*}
As well,
\begin{equation*}
    \hat\Phi z^{2(n+1)}e^{\hat\mu (n+1)^{1-2\gamma}}=\hat\Phi z^{2n}e^{\hat\mu n^{1-2\gamma}}z^2\l(1+O\l(\frac1{n^{2\gamma}}\r)\r)\rightarrow z^2.
\end{equation*}
It follows that $z^2=1$, which is a contradiction. Therefore
\begin{equation*}
    \widetilde\Phi(z)=\overline{\Phi(z)}.
\end{equation*}
\end{proof}

This completes the proof of the theorem.
\end{proof}

The following (final) theorem gives a formula of  the Weyl-Titchmarsh (or Kodaira) type for the spectral density.
It follows from asymptotics of orthogonal polynomials given by Theorem \ref{thm asymptotics of
polynomials} in a standard way (see \cite{Titchmarsh-46-1}, Chapter 5, and \cite{Kodaira-49}) and contains the Jost function,
which appears in the expression for the asymptotics of $P_n(\lambda)$.

    \begin{thm}\label{thm spectral density}
    Let $\{b_n\}$ be given by \eqref{eq entries} and the condition \eqref{eq conditions} hold.
    Then the spectrum of the Jacobi matrix $\mathcal J$ given by \eqref{eq definition of the operator}
    is purely absolutely continuous on $(-2;2)\backslash\{\pm2\cos\omega,\pm2\cos2\omega\}$,
    and for a.a. $\lambda\in(-2;2)$ the spectral density of $\mathcal J$ equals:
    \begin{equation}\label{eq spectral density}
        \rho'(\lambda)=\frac{\sqrt{4-\lambda^2}}{2\pi\l|F\l(\frac{\lambda}2-i\frac{\sqrt{4-\lambda^2}}2\r)\r|^2}
    \end{equation}
    (Weyl-Titchmarsh type formula), where the Jost function $F$ is defined in Theorem \ref{thm asymptotics of polynomials}.
    The denominator in \eqref{eq spectral density} does not vanish for $\lambda\in(-2;2)\backslash\{\pm2\cos\omega,\pm2\cos2\omega\}$.
    \end{thm}

\begin{proof}
If we rewrite \eqref{eq spectral density} in terms of the variable $z$, it reads:
\begin{equation}\label{eq spectral density in terms of z}
    \rho'\l(z+\frac1z\r)=\frac{1-z^2}{2\pi iz|F(z)|^2}.
\end{equation}

Polynomials of the second kind have asymptotics of the same type as polynomials of the first kind.
Indeed, define the cropped Jacobi matrix $\mathcal J_1$ as the original matrix $\mathcal J$ with the
first row and the first column removed,
\begin{equation*}
    \mathcal J_1=
    \left(%
    \begin{array}{cccc}
    b_2 & 1 & 0 & \cdots \\
    1 & b_3& 1 & \cdots \\
    0 & 1 & b_4 & \cdots \\
    \vdots & \vdots & \vdots & \ddots \\
    \end{array}%
    \right),
\end{equation*}
Polynomials of the second kind $Q_n(\lambda)$ associated to $\mathcal J$ are the polynomials of the first kind for $\mathcal J_1$.
Matrix $\mathcal J_1$ also satisfies conditions of Theorem \ref{thm asymptotics of polynomials}, which yields
that there exists a function $F_1$ analytic in $\mathbb D\backslash\{0\}$ and continuous in $\overline{\mathbb D}\cap U$ such that for $z\in\mathbb D\backslash\{0\}$,
\begin{multline}\label{eq asymptotics of Q hyperbolic}
    Q_n\l(z+\frac1z\r)=\frac{zF_1(z)}{1-z^2}\frac1{z^n}\exp\l(\frac{\mu_2(z)n^{1-2\gamma}}{1-2\gamma}\r)
    \\
    +o\l(\frac1{|z|^n}\exp\l(\frac{\Re\mu_2(z)n^{1-2\gamma}}{1-2\gamma}\r)\r)\as\ninf.
\end{multline}
This and asymptotics \eqref{eq asymptotics of P hyperbolic} of $P_n$ imply that the combination $Q_n(\lambda)+m(\lambda)P_n(\lambda)$ belongs
to $l^2$ for $\lambda\in\C_+$ and $\lambda\in\C_-$ only if
\begin{equation}\label{eq m}
    m\l(z+\frac1z\r)=-\frac{F_1(z)}{F(z)}\text{ for }z\in\mathbb D\backslash(-1;1).
\end{equation}
It follows that zeros of $F$ in $\mathbb D$ correspond to eigenvalues of $\mathcal J$ outside the interval $[-2;2]$
and hence can only lie on the interval $(-1;1)$. For every $z\in\mathbb T\cap U$, \eqref{eq m} has a limit,
\begin{equation*}
    m\l(z+\frac1z+i0\r)=-\frac{F_1(z)}{F(z)}.
\end{equation*}
This is what we look for, because
\begin{equation*}
    \rho'\l(z+\frac1z\r)=\frac1{\pi}\Im\, m\l(z+\frac1z+i0\r)=\frac{F(z)\overline{F_1(z)}-\overline{F(z)}F_1(z)}{2\pi i|F(z)|^2},
\end{equation*}
and the spectrum of $\mathcal J$ is purely absolutely continuous on intervals $(-2;2)\backslash\{\pm2\cos\omega,\pm2\cos2\omega\}$
(from the fact that the limit is finite at every point of these intervals, \cite{Gilbert-Pearson-87},\cite{Khan-Pearson-92}).

Theorem \ref{thm asymptotics of polynomials} also gives for $z\in\mathbb T\cap U$,
\begin{multline}\label{eq asymptotics of Q elliptic}
    Q_n\l(z+\frac1z\r)=\frac{zF_1(z)}{1-z^2}\frac1{z^n}\exp\l(\frac{\mu_2(z)n^{1-2\gamma}}{1-2\gamma}\r)
    \\
    +\frac{z\overline{F_1(z)}}{z^2-1}z^n\exp\l(-\frac{\mu_2(z)n^{1-2\gamma}}{1-2\gamma}\r)
    +o(1)\as\ninf.
\end{multline}
Substituting this expression and the analogous one \eqref{eq asymptotics of P elliptic} for $P_n$ into the formula for the Wronskian of $P$ and $Q$
(which is constant and equals to one), we get after a short calculation:
\begin{multline*}
    1=W\l(P\l(z+\frac1z\r),Q\l(z+\frac1z\r)\r)
    \\
    =P_n\l(z+\frac1z\r)Q_{n+1}\l(z+\frac1z\r)-P_{n+1}\l(z+\frac1z\r)Q_n\l(z+\frac1z\r)
    \\
    =\frac{z(F(z)\overline{F_1(z)}-\overline{F(z)}F_1(z))}{1-z^2},
\end{multline*}
where terms $o(1)$ cancel each other, so that the result does not depend on $n$. From this we have:
\begin{equation*}
    F(z)\overline{F_1(z)}-\overline{F(z)}F_1(z)=\frac1z-z,
\end{equation*}
which together with \eqref{eq spectral density in terms of z} gives \eqref{eq spectral density}.
\end{proof}

    \begin{cor}
    Under conditions of Theorem \ref{thm spectral density},
    \begin{equation*}
        \rho'(\lambda)=\frac{\sqrt{4-\lambda^2}}{2\pi\lim\limits_{\ninf}|P_{n+1}(\lambda)-zP_n(\lambda)|^2}\text{ for a.a. }\lambda\in(-2;2).
    \end{equation*}
    \end{cor}

\begin{proof}
From \eqref{eq relation between phi and P}, \eqref{eq phi asymptotics elliptic case}, Lemma \ref{lem conjugation} and \eqref{eq F} we have: for $z\in\mathbb T\cap U$,
\begin{multline*}
\left(%
\begin{array}{c}
  \l(P_{n+1}\l(z+\frac1z\r)-zP_n\l(z+\frac1z\r)\r)z^n\exp\l(-\frac{\mu_2(z)n^{1-2\gamma}}{1-2\gamma}\r) \\
  \l(P_{n+1}\l(z+\frac1z\r)-\frac1zP_n\l(z+\frac1z\r)\r)\frac1{z^n}\exp\l(\frac{\mu_2(z)n^{1-2\gamma}}{1-2\gamma}\r) \\
\end{array}%
\right)
\\
=
\left(%
\begin{array}{c}
  F(z) \\
  \overline{F(z)} \\
\end{array}%
\right)
+o(1)\as\ninf.
\end{multline*}
Therefore
\begin{equation*}
    |P_{n+1}(\lambda)-zP_n(\lambda)|\rightarrow|F(z)|\as\ninf,
\end{equation*}
and together with \eqref{eq spectral density} we obtain the assertion of the corollary.
\end{proof}

\section{The case $\gamma\in\l(\frac12;1\r]$}

In this section we formulate the result for the simpler case $\gamma\in\l(\frac12;1\r]$ and show how to simplify and modify
the proof of the corresponding result for $\gamma\in\l(\frac13;\frac12\r)$ to fit this formulation. We will need this
as a step in proving the asymptotics of the spectral density \eqref{eq forthcoming}.

    \begin{thm}\label{thm result in simpler case}
    Let $\{b_n\}_{n=1}^{\infty}$ be given by \eqref{eq entries} and
    \begin{equation*}
        \gamma\in\l(\frac12;1\r],\omega\notin\pi\mathbb Z\text{ and }\{q_n\}_{n=1}^{\infty}\in l^1.
    \end{equation*}
    Then for every $z\in\mathbb T\backslash\{1,-1,e^{\pm i\omega},-e^{\pm i\omega}\}$ there exists $F(z)$ such that
    orthonormal polynomials $P_n$ associated to $\{b_n\}_{n=1}^{\infty}$ have the following asymptotics:
    \begin{equation*}
    P_n\l(z+\frac1z\r)=\frac{zF(z)}{1-z^2}\cdot\frac1{z^n}+\frac{z\overline{F(z)}}{z^2-1}\cdot z^n+o(1)\as\ninf.
    \end{equation*}
    Function $F$ is continuous in $\mathbb T\backslash\{1,-1,e^{\pm i\omega},-e^{\pm i\omega}\}$ and does not have zeros there.
    Spectrum of the Jacobi matrix $\mathcal J$ given by \eqref{eq definition of the operator}
    is purely absolutely continuous on $(-2;2)\backslash\{\pm2\cos\omega\}$,
    and for a.a. $\lambda\in(-2;2)$ the spectral density of $\mathcal J$ equals:
    \begin{equation*}
        \rho'(\lambda)=\frac{\sqrt{4-\lambda^2}}{2\pi\l|F\l(\frac{\lambda}2-i\frac{\sqrt{4-\lambda^2}}2\r)\r|^2}.
    \end{equation*}
    \end{thm}

\begin{proof}
Consider $\gamma\in\l(\frac12;1\r]$ and return to previous sections. Statement of Lemma \ref{lem T-n} holds true if we
replace $U$ with $\mathbb C\backslash\{0,1,-1,e^{\pm\omega},-e^{\pm\omega}\}$, the estimate \eqref{eq estimate} with
\begin{equation*}
    \|R_n^{(2)}(z)\|=O\l(\frac1{n^{2\gamma}}+|q_{n+1}|\r)\as\ninf.
\end{equation*}
and $\mu_2(z)$ with zero. In the proof of Lemma \ref{lem T-n}, we include terms of the order $O\l(\frac1{n^{2\gamma}}\r)$ into the remainder and hence
make no use of $M_{\pm4}(z),X_{\pm4}(z),V(z)$ and $T_n^{(2)}(z)$. Condition
\begin{equation*}
    e^{\pm4i\omega}\neq1
\end{equation*}
is not needed anymore. System \eqref{eq system for w} becomes
\begin{equation*}
    w_{n+1}=\l[
    \left(%
        \begin{array}{cc}
        \frac1z & 0
        \\
        0 & z
        \end{array}%
    \right)
    +R_n^{(2)}(z)\r]w_n.
\end{equation*}
In statement and proof of Theorem \ref{thm asymptotics of polynomials} we can also replace $U$, $O\l(\frac1{n^{3\gamma}}+|q_{n+1}|\r)$ and $\mu_2(z)$
with correspondingly $\mathbb C\backslash\{0,1,-1,e^{\pm\omega},-e^{\pm\omega}\}$, $O\l(\frac1{n^{2\gamma}}+|q_{n+1}|\r)$ and $0$.
Most of the estimates that we use there become trivial. The same corrections should be applied to the proof of Theorem \ref{thm spectral density},
as well as replacing $(-2;2)\backslash\{\pm2\cos\omega,\pm2\cos2\omega\}$ with $(-2;2)\backslash\{\pm2\cos\omega\}$.
\end{proof}

\section{Acknowledgements}
The second author expresses his deep gratitude to Prof. S.N. Naboko for
his constant attention to this work and for many fruitful
discussions of the subject, to Dr. R. Romanov for his help and interesting discussions and to Prof. V. Bergelson
for the detailed explanations of the related subjects. The work was supported by
grants RFBR-09-01-00515-a  and INTAS-05-1000008-7883.

\end{document}